\documentclass[a4paper,10pt]{amsproc}

\usepackage{amsfonts, amsmath, amssymb, graphicx, mathrsfs}
\usepackage[all]{xy}

\newdir{ >}{!/6pt/@{ }*:(1,0.2)@_{>}*:(1,-0.2)@^{>}}

\theoremstyle{plain}

\newtheorem{theorem}{Theorem}[section]
\newtheorem{lemma}[theorem]{Lemma}

\theoremstyle{definition}
\newtheorem{definition}[theorem]{Definition}

\newtheorem{counter example}[theorem]{Counter Example}

\newtheorem{example}[theorem]{Example}

\numberwithin{equation}{section}

\author[S. Bag]{Sagarmoy Bag}
\address{Department of Pure Mathematics, University of Calcutta, 35, Ballygunge Circular Road, Kolkata 700019, West Bengal, India}
\email{sagarmoy.bag01@gmail.com}%

\author[R.C.Manna]{Ram Chandra Manna}
\address{Ramakrishna Mission Vidyamandir, Belur Math, Howrah-711202, West Bengal, India}
\email{mannaramchandra8@gmail.com}

\author[A. B. Raha]{Asit Baran Raha}
\address{Indian Statistical Institute, 203 Barrackpore Trunk Road, Kolkata-700 108, West Bengal, India}
\email{rahaab@gmail.com}

\thanks{ The second author thanks to Government of West Bengal, India, for financial support.}


\begin{document}

\title[ $CCS$-normal spaces]{ $CCS$-normal spaces }



\thanks {}
\maketitle	

\begin{abstract}
A space $X$ is called $CCS$-normal space if there exist a normal space $Y$ and a bijection $f: X\mapsto Y$ such that $f\lvert_C:C\mapsto f(C)$ is homeomorphism for any cellular-compact subset $C$ of $X$. We discuss about the relations between $C$-normal, $CC$-normal, $Ps$-normal spaces with $CCS$-normal.
\end{abstract}

\section{Introduction}
A topological space $X$ is called cellular-compact space if for any disjoint family $\mathcal{U}$ of non-empty open subsets of $X$, there exists a compact subset $K$ of $X$ such that $K\cap U\neq \emptyset$, for any $U\in \mathcal{U}$. The notion of the cellular-compact space was first introduced in \cite{T2019}.

A space $X$ is called $CCS$-normal (resp. $C$-normal, $CC$-normal, $Ps$-normal) space if there exist a normal space $Y$ and a bijection $f: X\rightarrow Y$ such that $f\lvert_C: C\rightarrow f(C)$ is a homeomorphism for any cellular-compact (resp. compact, countably compact, pseudocompact) subspace of $X$. The notion of $C$-normality is defined in \cite{AK2017}. The authors of \cite{AK2017} mentioned that $C$-normality was first introduced by Arhangel'skii while presenting a talk in a seminar held at Mathematics Department, King Abdulaziz University at Jeddah, Saudi Arabia in 2012. $CC$-normal spaces and $Ps$-spaces are initiated in \cite{KA2017} and \cite{B2019} respectively. $CCS$-normal is a generalization of normal space. Every $Ps$-normal space is $CCS$-normal and every $CCS$-normal space is $C$-normal. We give an example of a $C$-normal space which is not $CCS$-normal [Example \ref{CCbutnotCCS}]. Also we give an example of a space which is $CCS$-normal but not $Ps$-normal [ Example \ref{CCSbutnotPs}].
Finally we conclude this article by raising some open problems related to $CCS$-normal spaces.

\section{$CCS$-normal spaces}

\begin{definition}
 A space $X$ is called $CCS$-normal if there exist a normal space $Y$ and a bijection $f:X\rightarrow Y$ such that $f|_C:C\rightarrow f(C)$ is a homeomorphism for every cellular-compact subset $C$ of $X$.
\end{definition}

\begin{theorem}
Every $CCS$-normal space is a $C$-normal space.
\end{theorem}

\begin{proof}
Let $X$ be a $CCS$-normal space. Then there exist a normal space $Y$ and a bijection $f:X\rightarrow Y $ such that $f|_C:C\rightarrow f(C)$ is a homeomorphism for every cellular-compact subset $C$ of $X$. Let $C$ be a compact subset of $X$. Then $C$ is cellular-compact. Hence $f|_C:C\rightarrow f(C) $ is a homeomorphism. Therefore $X$ is a $C$-normal space.
\end{proof}

In Proposition 3.1,\cite{T2019}, the authors proved that every Tychonoff Cellular-Compact space is pseudocompact. But we prove that the result holds anyway. We do not need Tychonoffness property.

\begin{theorem}\label{CCisPS}
Any cellular-compact space is pseudocompact.
\end{theorem}

\begin{proof}
Let $X$ be a Cellular-Compact space. Let $\mathcal{U}$ be an infinite family of non-empty disjoint open sets. There exists a compact set $K\subset X$ such that $K\cap U \neq \emptyset$ for $U\in \mathcal{U}$. Choose $x_{U}\in K\cap U, U\in \mathcal{U}$. The infinite set $\{x_{U} : U\in \mathcal{U}\}$ is contained in $K$, a compact set. Hence $\{x_{U} : U\in \mathcal{U} \}$ has an accumulation point, say $x\in K$. Then the family $\mathcal{U}$ has an accumulation point. Therefore $X$ is feebly compact which implies pseudocompact.     	
\end{proof}	

\begin{theorem}\label{PSimpliesCCS}
Any $Ps$-normal space is $CCS$-normal.
\end{theorem}

\begin{proof}
Let $X$ be a $Ps$-normal space. Then there exist a normal space $Y$ and a bijection $f:X\rightarrow Y$ such that $f|_P:P\rightarrow f(P)$ is a homeomorphism for every pseudocompact subset $P$ of $X$. Let $C$ be a cellular-compact subset of $X$. So, $C$ is a pseudocompact subset of $X$ by Theorem \ref{CCisPS}. Since $X$ is $Ps$-normal, $f|_C:C\rightarrow f(C)$ is a homeomorphism. Hence $X$ is a $CCS$-normal space.	
\end{proof}

In $[2]$ the authors have shown that the Dieudonne Plank is $C$-normal. In what follows we can demonstrate that it is, indeed,$Ps$-normal. Let $X$ be the Dieudonne plank i.e, $X= [0,\omega_{1}]\times [0,\omega_{0}]-(\omega_{1},\omega_{0})$ with the topology described bellow:

 Write $X=A\cup B\cup N$ where $A= \{\omega_{1}\}\times \mathbb{N}, B= [0, \omega_{1})\times \{\omega_{0}\}, N=[0,\omega_{1})\times \mathbb{N}$, where $\mathbb{N}$ includes $0$ also. All points of $N$ are isolated. Basic open neighbourhoods of $ (\omega_{1}, n)\in A $ are of the form $U_{n}(\alpha) = (\alpha, \omega_{1}]\times \{n\},\alpha<\omega_{1}, n\in \mathbb{N}$. Basic open neighbourhoods of $(\alpha, \omega_{0})\in B$ are of the form $V_{\alpha}(n)= \{\alpha\}\times (n,\omega_{0}], n\in \mathbb{N}, \alpha<\omega_{0}$. $A$ and $B$ are closed subsets of $X$. Any horizontal line $ B_{n}= [0,\omega_{1}]\times \{n\}$ is clopen set,$n\in \mathbb{N}$. No $B_{n}$ is pseudocompact: $B_{n} = [0,\omega_{1}]\times \{n\}$. Let $U_{n}(\alpha)= (\alpha,\omega_{1}]\times \{n\}$  be an open neighbourhood of $(\omega_{1}, n)\in B_{n}$. Since $\alpha <\omega_{1}$ is an infinite countable ordinal, $[0,\alpha)$ is a countably infinite set . We can write $[0,\alpha) = \{0=\alpha_{1}<\alpha_{2}<\alpha_{3}<\dots <\alpha_{m}<\dots \}$. Define $f:B_{n}\rightarrow \mathbb{R}$ by $f(\beta) = 0, \alpha\leq \beta\leq \omega_{1}, = m $ if $\beta= \alpha_{m}, m\in \mathbb{N}$. Clearly $f$ is continuous and unbounded. This shows that $X$ is not pseudocompact; for if $X$ is pseudocompact, each $B_{n}$ , being clopen subset , would be pseudocompact-which is false. Suppose $C\subset X$ is pseudocompact. Then $C\cap B_{n}$, for each $n\in\mathbb{N}$ is clopen subset of $C$ and hence, pseudocompact. Now $C\cap B_{n}$ consists of either isolated points of $B_{n}$ or atmost one non isolated point, viz, $(\omega_{1}, n)$. Hence $C\cap B_{n}$ cannot be pseudocompact if infinite. So $C\cap B_{n}$ is either empty or a non-empty finite set for each $n\in\mathbb{N}$. Now take $Y = X$ as set with a topology described by the following nbhood system : All points of $N\cup A = [0, \omega_{1}]\times \mathbb{N}$ are isolated. The neighbourhoods of the points $(\alpha, \omega_{0})$ are unaltered. The topology of $Y$ is $T_{4}$, first-countable and is finer than the topology of $X$. Let $C$ be a pseudocompact subset of $X$. Let $(\alpha, \upsilon)$ be a typical point of $X, \alpha\in [0,\omega_{1}], \upsilon\in [0, \omega_{0}]$. Let $(\alpha, \upsilon)\in C$ and $1_{X}:X\rightarrow Y$ be the identity map. If $(\alpha, \upsilon )\in N,$ say $(\alpha, n), (\alpha, n)$ is isolated both in $X$ and $Y$. Then $1_{X}$ is continuous at $(\alpha,n)$. If $(\alpha, \upsilon)\in B$, say $(\alpha, \omega_{0})$, the $(\alpha,\omega_{0})$ has the same nbhood in $X$ and $Y$. $1_{X}$ is continuous at $(\alpha, \omega_{0})$ as well.

If $(\alpha,\upsilon)\in A$ i.e, $(\omega_{1},n)\in A$. Then $(\omega, n)\in C\cap B_{n}$. Now $C\cap B_{n}$ is a finite set containing $(\omega_{1},n )$. Let $C\cap B_{n}-(\omega_{1},n)=\{(\alpha_{1},n),(\alpha_{2},n),\dots (\alpha_{k},n)\}$, if non-empty. Choose $\beta<\omega_{1}$ such that $\alpha_{i}<\beta, 1\leq i\leq k. (\beta, \omega_{1}]\times \{n\}$ is a neighbourhood of $(\omega_{1},n) $ in $B_{n}$ and $C\cap (\beta, \omega_{1}]\times \{n\} = \{(\omega_{1},n)\}$ is an open nbhood of $(\omega_{1},n)$ in $C$ as a subspace of $X$. Since $\{(\omega_{1},n)\}$ is an open nbhood of $(\omega_{1}, n)$ in $Y$, $1_{X}$ is continuous at $(\omega_{1},n)$. So we come to the conclusion that $1_{X}:C\rightarrow 1_{X}(C)\subset Y$ is a continuous map. So $1_{X}(C)$ is pseudocompact in $Y$. Now the topology of $Y$ is finer than that of $X$ , the identity map $1_{Y}:Y\rightarrow X$ is continuous. This implies that $1_{Y}:1_{X}(C)\rightarrow 1_{Y}(1_{X}(C))= C$ is continuous. We can now conclude that $1_{X}:C\rightarrow 1_{X}(C)$ is a homeomorphism for each pseudocompact subset $C\subset X$. $X$ is then $Ps$-normal and by Theorem \ref{PSimpliesCCS} it is $CCS$-normal..

\begin{theorem}
If $X$ is $CC$-normal, first-countable, regular space. Then $X$ is a $CCS$-normal space.
\end{theorem}

\begin{proof}
Since $X$ is a $CC$-normal space, then there exist a normal space $Y$ and a bijection $f:X\rightarrow Y$ such that $f|_C:C\rightarrow f(C)$ is a homeomorphism, for every countably-compact subset $C$ of $X$. Take any cellular-compact subset $D$ of $X$. Then $D$ is a cellular-compact, first countable and regular space. So $D$ is countably compact, [Corollary $4.2$, \cite{T2019}]. Hence $f|_D:D\rightarrow f(D)$ is a homeomorphism. So, $X$ is a $CCS$-normal. 	
\end{proof}

\begin{theorem}
$CCS$-normality is a topological property.
\end{theorem}

\begin{proof}
Let $X$ be a $CCS$-normal space and $X$ is homeomorphic to $Z$. Let $Y$ be a normal space and $f:X\rightarrow Y$ be a bijective mapping such that $f|_C:C\rightarrow f(C)$ is a homeomorphism for any cellular-compact subspace $C$ of $X$. Let $g:X\rightarrow Z$ be the homeomorphism. Then $f\circ g^{-1}:Z\rightarrow Y$is the required map.  	
\end{proof}

\begin{theorem}\label{notCCS}
If $X$ is cellular-compact but not normal, then $X$ is not $CCS$-normal.
\end{theorem}

\begin{proof}
Let $X$ be cellular-compact but not normal. On the contrary assume that there exist a normal space $Y$ and a bijective mapping $f:X\rightarrow Y$ such that $f|_C:C\rightarrow f(C)$ is a homeomorphism, for any cellular-compact subspace $C$ of $X$. In particular $f:X\rightarrow Y$ is a homeomorphism. This is a  contradiction as $Y$ is normal and $X$ is not normal. Hence $X$ is not $CCS$-normal.	
\end{proof}

The following example is an example of a $CC$-normal space which is not $CCS$-normal.

\begin{example}\label{CCbutnotCCS}
Consider $(\mathbb{R},\tau_c)$, where $\tau_c$ is the co-countable topology on $\mathbb{R}$. Then $(\mathbb{R},\tau_c)$ is $CC$-normal,[ Example 2.6, \cite{KA2017} ]. Since there is no disjoint family of open sets in $(\mathbb{R},\tau_c)$, so $(\mathbb{R}.\tau_c)$ is cellular-compact (vacuously). It is known that $(\mathbb{R}, \tau_c)$ is not normal. Hence by Theorem \ref{notCCS}, $(\mathbb{R},\tau_c)$ is not $CCS$-normal. Also every $CC$-normal space is $C$-normal. Therefore $(\mathbb{R},\tau_c)$ is $C$-normal but not $CCS$-normal.
\end{example}

Now we give an example of a $CCS$-normal space which is not $Ps$-normal.

\begin{example}\label{CCSbutnotPs}
Let $X= [0,\omega_{1})\times [0,\omega_{1}]$ be equipped with the product of the respective order topologies. Then

\begin{itemize}
\item[1)] $X$ is countably compact , because $\Omega=[o,\omega_{1})$ is countably compact and $\Omega_{0}=[0,\omega_{1}]$ is compact.

\item[2)] $X$ is locally compact, because $\Omega$ is locally compact and $\Omega_{0}$ is compact.

\item[3)] $X$ is not normal.
\end{itemize}

Following steps are needed to prove that $X$ is not normal.

\begin{itemize}
 \item[a)] Interlacing Lemma: Let $\{\alpha_{n}, n\in \mathbb{N}\}$ and $\{\beta_{n}, n\in \mathbb{N}\} $ be two sequences in $\Omega=[0,\omega_{1})$ such that $\alpha_{n}\leq \beta_{n}\leq \alpha_{n+1}$ for each $n\in \mathbb{N}$ . Then both sequence converge and to the same point of $\Omega$

 \item[b)] If $f:\Omega \mapsto \Omega$ is a function such that $\alpha \leq f(\alpha)$ for each $\alpha$, then for some $\alpha \in \Omega$, the point $(\alpha,\alpha)$ is an accumulation point of the graph $G(f)= \{(\alpha, f(\alpha)): \alpha\in \Omega\}$ of $f$.
\end{itemize}

\begin{proof}
Let $\alpha$ be any element of $\Omega$. Let $\alpha_{1} = \alpha, \alpha_{2}= f(\alpha_{1}),......\alpha_{n+1}= f(\alpha_{n}) , n\in \mathbb{N}$. Now $\alpha_{1}\leq f(\alpha_{1})= \alpha_{2}\leq f(\alpha_{2})\leq \alpha_{3}.....$. Let $\beta_{n}= f(\alpha_{n}), n\in \mathbb{N}$. Applying Interlacing Lemma to sequences $\{\alpha_{n}, n\in\mathbb{N}\}, \{\beta_{n}, n\in	\mathbb{N}\}$ we obtain a point $\alpha_{0}\in \Omega$. Consider the point $(\alpha_{0},\alpha_{0})\in X$ and any open nbhod of $(\alpha_{0},\alpha_{0})$ of the form $(\gamma,\alpha_{0}]\times (\gamma,\alpha_{0}]$, where $0\leq \gamma \leq \alpha_{0}$. Since $\alpha_{0}= \lim\alpha_{n}= \lim \beta_{n}$, there exists $n_{0}\geq 1$ such that $\gamma < \alpha_{n}\leq \alpha_{0}, \gamma < \beta_{n}\leq \alpha_{0}$ for $n\geq n_{0}$. Then $(\alpha_{n},f(\alpha_{n}))= (\alpha_{n},\beta_{n})\in (\gamma,\alpha_{0}]\times (\gamma,\alpha_{0}]$ for $n\geq n_{0}$. This shows that $(\alpha_{0},\alpha_{0})$ is an accumulation point of $G(f)= \{(\alpha, f(\alpha)): \alpha\in \Omega\}$. Now in order to show that $X$ is not normal, let $A= \{(\alpha,\alpha): \alpha\in \Omega\}$ and $B= [0,\omega_{1})\times \{\omega_{1}\}$  be two disjoint closed subsets of $X$ . Assume $U$ is an open nbhood of $A$ and $V$ is an open nbhod of $B$ and $U\cap V=\emptyset$.

Claim: For any $\alpha\in \Omega$ there exists $\alpha<\beta\in\Omega$ such that $(\alpha,\beta)\notin U$. If not there exists $\alpha_{0}\in \Omega$ such that $(\alpha_{0},\beta)\in U$ for all $\beta>\alpha_{0}$. Now $(\alpha_{0},\omega_{1})\in B$. There exists $\gamma, \delta \in \Omega$ such that $ (\alpha_{0},\omega_{1})\in (\gamma, \alpha_{0}]\times (\delta,\omega_{1}]\subset V $, where $\gamma<\alpha_{0}, \delta<\omega_{1}$. Choose $\beta<\omega_{1}$ such that $\alpha_{0}<\beta,\delta<\beta$. Now $(\alpha_{0},\beta)\in (\gamma,\alpha_{0}]\times (\delta,\omega_{1}]\subset V$. Since $\alpha_{0}<\beta, (\alpha_{0},\beta)\in U$ also $(\alpha_{0},\beta)\in V$. i.e, $U\cap V\neq \emptyset$, contradiction. This settles the claim. Define $f:\Omega\mapsto \Omega$ by $f(\alpha)= min\{\beta\in \Omega: \alpha<\beta , (\alpha,\beta)\notin U\}$. Then $\alpha<f(\alpha)$ for each $\alpha\in \Omega $. By $(b)$ there exists $\alpha_{0}\in\Omega $such that $(\alpha_{0},\alpha_{0})$ is an accumulation point of $G(f)= \{(\alpha,f(\alpha)): \alpha\in \Omega\}\subset U^{c}$ i.e,$(\alpha_{0},\alpha_{0})\in \overline{G(f)}\subset U^{c}$ since $U^{c}$ is closed. Contradiction because $(\alpha_{0},\alpha_{0})\in A\subset U$, we conclude that $X$ is not normal. Because of $(1)$ $X$ is not $CC$-normal and because of $(2)$ $X$ is $C$-normal. As $X$ is countably compact, $X$ is Pseudocompact but not normal. So $X$ is not PS-normal.

$(4)$:   $X$ is $CCS$-normal: $X$ is not cellular-compact. This is because $X$ has a dense subset of isolated points. Let $A\subset X$ be a cellular-compact subspace. Let $p_{1}:X\mapsto \Omega, p_{2}:X\mapsto \Omega_{0}$ be two projections.Then $A\subset p_{1}(A)\times p_{2}(A)$. Now $p_{1}(A)$ is a cellular-compact in $\Omega$. We claim that $p_{1}(A)$ is contained in a compact subset.

Case$(i): p_{1}(A)$ is bounded set in $\Omega$. Then for $\alpha<\omega_{1}, p_{1}(A)\subset[0,\alpha]$ and $[0,\alpha ] $ is a compact subspace.

Case$(ii): p_{1}(A)$ is unbounded . Then $p_{1}(A)$ contains a dense subspace of isolated points of $p_{1}(A)$ and must be compact . But unbounded subset of $\Omega$ cannot be compact. Assume $p_{1}(A)\subset K $ a compact subset of $\Omega$. $A\subset p_{1}(A)\times p_{2}(A)\subset K\times \Omega_{0}$ and $K\times \Omega_{0}$ is compact in $X$. Since $X$ is $C$-normal and every cellular-compact subset is contained in a compact subspace , $X$ becomes $CCS$-normal.

 Thus we obtain $X$ as an example of a $CCS$-normal space which is not $Ps$-normal.	
\end{proof}	
\end{example}

\begin{theorem}
If $X$ is $CCS$-normal and Frechet, $f:X\rightarrow Y$ is witness of $CCS$-normality, Then $f$ is a continuous mapping.
\end{theorem}

\begin{proof}
It is sufficient to show that for any $A\subset X$, $f(\overline{A})\subset \overline{f(A)}$ . Let $\emptyset \neq A\subset X$ and  $y\in f(\overline{A})$. Let $x$ be the unique point in $\overline{A}$ such that $f(x)=y$. Since $X$ is Frechet, then there exists a sequence $\{x_{n}, n\geq 1\}$ in $A$ such that $x = \lim_{n\to \infty}x_{n}$ . Let $K=\{x_{n}:n\geq 1\}\cup \{x\}$ .Then $K$ is compact and therefore $K$ is cellular-compact. Then $f|_K:K\rightarrow f(K)$ is a homeomorphism, as $X$ is $CCS$-Normal. Let $U$ be any open set in $Y$ containing $y$ . Then $U\cap f(K)$ is open set in $f(K)$ containing $y$. Now $\{x_{n} : n\geq 1\}$ is dense in $K$ and is contained in $K\cap A\Longrightarrow K\cap A $ is dense in $K\Longrightarrow f(K\cap A)$ is dense in $f(K)\Longrightarrow f(K\cap A)\cap f(K)\cap U\neq \emptyset \Longrightarrow   U\cap f(A)\neq \emptyset$. Therefore $y\in \overline{f(A)}$. Hence $f(\overline{A})\subset \overline{f(A)}$.  	
\end{proof}

\begin{theorem}\label{sum}
Let $X=\bigoplus_{\alpha\in \Lambda}X_\alpha$ and $C\subseteq X$. Then $C$ is cellular-compact if and only if $\Lambda_\circ=\{\alpha\in \Lambda : C\cap X_\alpha\neq \emptyset\}$ is finite and $C\cap X_\alpha$ is cellular compact in $X_\alpha$ for each $\alpha\in \Lambda_\circ$.
\end{theorem}

\begin{proof}
Let $C$ be cellular-compact. If possible let $\Lambda_\circ$ be infinite set. Then $\{C\cap X_\alpha :\alpha\in \Lambda_\circ\}$ is a disjoint family of open sets in $C$. So there exists a compact set $K$ in $C$ such that $K\cap C\cap X_\alpha\neq \emptyset$ for all $\alpha \in \Lambda_\circ$. Thus $\{C\cap X_\alpha : \alpha \in\Lambda_\circ\}$ is an open cover of $K$ which has no finite subcover, which is a contradiction as $K$ is compact. Therefore $\Lambda_\circ$ is finite.

Let $\alpha\in \Lambda_\circ$. We show $C\cap X_\alpha$ is cellular-compact. Let $\{U_\gamma : \gamma\in I\}$ be a disjoint family of open sets in $C\cap X_\alpha$. Now each $U_\gamma$ is open in $C$ as $C\cap X_\alpha$ is open in $C$. Then there exists a compact subset$K$ of $C$ such that $K\cap U_\gamma\neq \emptyset$ as $C$ is cellular-compact. $K$ is compact in $C$ then $K\cap X_\alpha$ is compact in $C\cap X_\alpha$ and also $(K\cap X_\alpha)\cap U_\gamma\neq \emptyset$ for all $\gamma\in I$. Therefore $C\cap X_\alpha$ is cellular compact for all $\alpha\in \Lambda_\circ$.

Conversely, let the given condition hold. Let $\{U_\gamma :\gamma\in I\}$ be a disjoint family of open sets in $C$. For $\alpha\in \Lambda_\circ$, $\{U_\gamma\cap X_\alpha: \gamma\in I\}$, if non-empty, is a disjoint family of open sets in $C\cap X_\alpha$. This implies there exists a compact set $K_\alpha$ in $C\cap X_\alpha$ such that $K_\alpha \cap (U_\gamma\cap X_\alpha)\neq \emptyset$ for all $\gamma\in I$. Let $K=\bigoplus_{\alpha\in \Lambda_\circ}K_\alpha$. Then $K$ is a compact set in $C$ such that $K\cap U_\gamma\neq \emptyset$ for all $\gamma\in I$. Hence $C$ is cellular compact.
\end{proof}

\begin{theorem}
$CCS$-normality is an additive property.
\end{theorem}

\begin{proof}
Let $X = \bigoplus_{\alpha \in \Lambda}X_{\alpha}$ and each $X_{\alpha}$ is cellular-compact. Then for each $\alpha \in \Lambda$ there exists a normal space $Y_{\alpha}$ and a bijective mapping $f_{\alpha }:X_{\alpha}\rightarrow Y_{\alpha}$ such that $f_{\alpha}|_{C_{\alpha}}:C_{\alpha}\rightarrow f_{\alpha}(C_{\alpha})$ is a homeomorphism for each cellular-compact subspace $C_{\alpha}$ of $X_{\alpha}$. Then $Y =\bigoplus_{{\alpha}\in {\Lambda}}Y_{\alpha}$ is a normal space. Consider the function $f =\bigoplus_{\alpha\in \Lambda}f_{\alpha}:X\rightarrow Y$. Let $C$ be a cellular -compact
subspace of $X$. Then by Theorem \ref{sum}
, $\Lambda_\circ = \{\alpha \in \Lambda :C\cap X_{\alpha}\neq \emptyset\}$ is finite and $C\cap X_{\alpha}$ is cellular -compact for each $\alpha \in \Lambda$ Then $f|_{C}:C\rightarrow f(C)$ is a homeomorphism. Hence $X$ is $CCS$-normal.	
\end{proof}

\section{Some special results}

\begin{definition}
 A topological space $X$ is called submetrizable if the topology admits of a weaker metrizable topology.
\end{definition}

It is proved in \cite{AK2017} that submetrizable space is $C$-normal. In \cite{KA2017}, the authors raised a question. Is a submetrizable space $CC$-normal? We can claim that the answer is affirmative.

\begin{theorem}
Every submetrizable space is $CC$-normal.
\end{theorem}

In order to prove this we need the following Lemmas.

\begin{lemma}\label{closed}
Let $X$ be a first countable Hausdorff space. If $A\subset X$ is countably compact, then $A$ is closed in $X$.
\end{lemma}

\begin{proof}
Let $\emptyset \neq A \neq X$ be countably compact subspace. Let $\{U_{n} : n\geq 1\}$ be a countable local base at $x\notin A$. For each $y\in A$ there exist open sets $U_{n(y)}, V_{n(y)}$ such that $x\in U_{n(y)}, y\in V_{n(y)}$ and $U_{n(y)}\cap V_{n(y)}=\emptyset$. Now $\{V_{n(y)} : y\in A \}$ is , infact , a countable open cover of $A$ and hence , admits of a finite subcover, say $\{V_{n(y_{i})} : 1\leq i\leq k\}$. Look at $U_{x} = \bigcap_{i=1}^{k}U_{n(y_{i})}$. Then $U_{x}\cap A=\emptyset $ and $ x\in U_{x}$ an open neighbourhood of $x$ . Thus $X-A$ is open $\Longrightarrow A$ is closed.	
\end{proof}

\begin{lemma}\label{homeomorphism}
If $X$ be a countably compact space, $Y$ is first countable Hausdorff space and $f:X\longrightarrow Y$ is a continuous bijection. Then $f$ is a homeomorphism.
\end{lemma}

\begin{proof}
Since $f$ is a continuous bijection, suffices to show that $f$ is a closed map. Let $A\subset X$ be  closed. Then $A$ is countably compact and ipso facto, $f(A)$ is countably compact in $Y$. Since $Y$ is first countable , Hausdorff, $f(A)$ is closed in $Y$, i.e, $f$ is a closed map. Consequently $f$ is a homeomorphism.   	
\end{proof}

\begin{theorem}
Every submetrizable space is $CC$-normal.
\end{theorem}

\begin{proof}
Let $(X , \tau )$ be a submetrizable space. i.e, there exists a topology $\tau^{'}\subset \tau$ such that $(X , \tau^{'})$ is metrizable. Let $1_{X}:(X,\tau)\longrightarrow (X , \tau^{'})$ be the identity map which is continuous and bijective. $(X, \tau^{'})$ is normal being metrizable. Take any countably compact subspace $C\subset (X , \tau)$. $1_{X}(C)$ is countably compact in $(X , \tau^{'})$ and $1_{X}:C\longrightarrow 1_{X}(C)$ is a continuous , bijection. By Lemma \ref{homeomorphism}, $1_{X}$ is a homeomorphism, so that $(X, \tau)$ is CC-normal. Further note that $1_{X}(C)$ being a countably compact subset of a metrizable space , is indeed compact which renders $C$ into a compact subset of $(X,\tau)$. This yields the further information that in a submetrizable space ,  compact and countably compact subsets coincide. 	
\end{proof}

\begin{example}
Example \ref{notcellularcompact} produces an example of a submetrizable space which is not cellular-compact. It is proved in \cite{T2019}  that a cellular-compact , first-countable ,$T_{3}$ space is countably-compact and ipsofacto, closed in view of Lemma \ref{closed}. We can conclude that in a metrizable space all three concepts , viz, Compactness,countable-compactness and Cellular-compactness coincide.
\end{example}

\begin{definition}
Let $M$ be a non-empty proper subset of a topological space $( X,\tau)$. Define a new topology $\tau_{(M)}$ on $X$ as follows: $\tau_{(M)}= \{U\cup K: U\in \tau$ and $ K\subset X\setminus M\} , (X, \tau_{(M)})$ is called a discrete extension of $(X, \tau)$ and we denote it by $X_M$.		
\end{definition}

Discrete extension of a cellular-compact space need not be cellular compact. This can be shown in the following example.

\begin{example}\label{notcellularcompact}
Let $X = [0,1]$ with subspace topology of usual topology on $\mathbb{R}$ and $M= X\setminus \{\frac{1}{n} : n\in \mathbb{N}\}\cup \{0\}$. Then $X$ is compact $\Longrightarrow X$ is cellular-compact.\qquad\\Claim: $X_M$ is not cellular-compact. Let $\mathcal{U}= \{(\frac{1}{n+1},\frac{1}{n}) : n\in \mathbb{N}\}\cup \{\{\frac{1}{n}\} : n\in \mathbb{N}\}\cup \{0\}$ is a disjoint family of open sets in $X_M$. There does not exists any compact set $K$ in $X_M$ such that $K\cap U\neq \emptyset$, for all $U\in \mathcal{U}$. Therefore $X_M$ is not cellular-compact.
\end{example}

Let $X$ be any topological space. Let $X'=X\times \{1\}$. Let $A(X)=X\cup X'$. For $x\in X$ we denote $<x,1>$ as $x'$ and for $B\subseteq X$ let $B'=\{x':x\in B\}=B\times \{1\}$. Let $\{\beta (x): x\in X\}\cup \{\beta (x'):x'\in X'\}$ be neighbourhood system for some topology $\tau$ on $A(X)$, where for $x\in X, \beta (x)=\{ U\cup (U'\setminus\{x'\}): U$ is an open set in $X$ containing $x\}$ and for $x'\in X', \beta (x')=\{\{x'\}\}$. $(A(X),\tau)$ is called Alexandroff Duplicate of $X$.

OBSERVATION: Let $X$ be a topological space and $A(X) = X\cup X^{'}$ is the Alexandroff duplicate space. Let $C\subset A(X)$ be a compact subspace of $A(X)$. Let $C\cap X= D, C\cap X^{'}= E$. If $D=\emptyset $, then $E= C\subset X^{'} \Longrightarrow C$ is finite. Suppose $D\neq \emptyset$. For $x\in D$ let $U_{x}\cup U_{x}^{'}\setminus \{x^{'}\}$ be an open neighbourhood of $x$ . Consider the family $\{ U_{x}\cup U_{x}^{'}\setminus \{x^{'}\}: x^{'}\in D\}\cup \{ \{x'\}: x'\in E\}$ of open subsets of $A(X)$, which is an open cover of $C$. Let $\{U_{x_{i}}\cup U_{x_{i}}^{'}\setminus \{x_{i}^{'}\}: 1\leq i\leq n \}\cup \{\{y_{j}\}: y_{j}\in E, 1\leq j \leq m\}$ be a finite subcover of $C= D\cup E$. Then $C\cap X= D\subset \bigcup_{i=1}^{n}U_{x_{i}}= U\Longrightarrow U_{x_{i}}^{'}\setminus \{x_{i}^{'}\}\subset U_{x_{i}}^{'}\Longrightarrow \bigcup_{i=1}^{n}U_{x_{i}}^{'}\setminus \{x_{i}^{'}\}\subset \bigcup_{i=1}^{n}U_{x_{i}}^{'}= U^{'}.$ Hence $C= D\cup E\subset U^{'}\cup \{\{y_{j}\} :y_{j}\in E, 1\leq j\leq m\}\cup U\Longrightarrow E\subset U^{'}\cup \{\{y_{j}\}: y_{j}\in E, 1\leq j\leq m\}\Longrightarrow E-U^{'}$ is a finite subset of $X^{'}$. So, we conclude that: A subset $C\subset A(X) $ is compact $\Longrightarrow$ for any open set $U\supset $ compact set $C\cap X, C\cap X^{'}-U^{'}$ is a finite subset of $X^{'}$. Conversely let $C\subset A(X)$ be a subset such that $C\cap X$ is compact and for any open set $U\subset X , C\cap X\subset U, (C\cap X^{'})\setminus U^{'}$ is a finite subset of $X^{'}$.\qquad\\ Claim: $C$is compact. It suffices to prove that any open cover of $C$ by means of basic open sets of $A(X)$ has a finite subcover. Let $\{ U_{x}\cup U_{x}^{'}\setminus \{x^{'}\}: x\in C\cap X\}\cup \{\{x^{'}\}; x^{'}\in C\cap X^{'}\}$ be  an open cover of $C$. Since $C\cap X$ is compact and $\{U_{x}: x\in C\cap X\}$ is an open cover $C\cap X$ in $X$. There exists a finite subcover, say, $\{U_{x_{i}}: i\leq i \leq n\}$ ie, $C\cap X \subset \bigcup_{i=1}^{n}U_{x_{i}}= U $(say). By hypothesis $(C\cap X^{'})\setminus U^{'}$ is a finite subset of $X^{'}$, say, $\{y_{1},y_{2},\dots ,y_{m}\}$. Hence $C\cap X^{'}= U^{'}\cup \{y_{1},y_{2},\dots ,y_{m}\}$. Let $D= \{y_{1},y_{2},\dots ,y_{m}\}\cup \{x_{i}^{'}\in C\cap X^{'}: 1\leq i\leq n\}$. Now $\{U_{x_{i}}\cup U_{x_{i}}^{'}\setminus \{x_{i}^{'}\}: 1\leq i\leq n\}\cup \{\{x_{i}^{'}\}: x_{i}^{'}\in C\cap X^{'}, 1\leq i\leq n\}\cup \{\{y_{j}\}: 1\leq j\leq m\}$ is a finite open subcover for $C$ from $\{U_{x}\cup U_{x}^{'}\setminus \{x^{'}\}: x\in C\cap X\}\cup \{\{x_i^{'}\} : x_i^{'}\in C\cap X^{'}\}$. Hence $C$ is compact. \qquad\\ We come to the conclusion that $C\subset A(X)$ is compact if and only if every open set $U\subset X, C\cap X\subset U$ we must have $C\cap X^{'}\setminus U^{'}$ is a finite subset of $X^{'}$.\qquad\\ Using this result we can cite an example for which $A(X)$ is not cellular-compact but $X$ is cellular-compact.

\begin{example}
Suppose $A(X)$ is Cellular-Compact and $X$ is infinite. Now isolated points of $A(X)= X\cup X^{'}$, namely the points of $X^{'}$  form an infinite family of disjoint open sets. Hence $A(X)$ must contain a compact set $C$ such that $C\cap \{x^{'}\}\neq \emptyset$ for each $x^{'}\in X^{'}$. ie, $C\supset X^{'}$. Since $C= X^{'}$ is impossible, $C\cap X\neq \emptyset$ and $C\cap X$ is a compact set. As $C$ is compact in $A(X)$, by the preceding observation, if $U\subset X$ is an open set containing $C\cap X$, then $X^{'}\setminus U'= C\cap X^{'}\setminus U^{'}$ is a finite subset of $X^{'}$ vide Theorem 2.2.\cite{AK2017}. i.e, $X\setminus U$ is a finite subset $X$. i.e, every open set of $X$ containing $C\cap X$ is cofinite in $X$.\qquad\\

Now take $X= \mathbb{R}$ with cocountable topology . Then $X$ is a $T_{1}$ space which is cellular-compact. Also any compact subset of $X$ is finite. Now look at $A(X)= X\cup X^{'}$. If $A(X)$ has to be cellular-compact, from the preceding para it is clear that $X$ must have a compact set $K$ such that any open set containing $K$ must be cofinite. Since $K$ is anyhow finite, this is not true! Thus $A(X)$ cannot be cellular-compact.
\end{example}
	
\begin{theorem}
Let $X$ be a Hausdorff space and $A(X)= X\cup X^{'}$ be the Alexandroff duplicate space. If $A(X)$ is cellular-compact , then there should exist a compact set $K\supset$ the dense set of isolated points of $X^{'}$. Since $K$ is closed , as $A(X)$ is $T_{2}, K= A(X)$. Hence $A(X)$ is compact. Consequently if $X$ is a non-compact $T_{2}$ space , then $A(X)$ cannot be cellular-compact. In presence of Hausdorffness $A(X)$ is cellular-compact if and only if $X$ is compact if and only if $A(X)$ is compact.
\end{theorem}

The following problems are open.

\begin{itemize}
\item[1)] Is the discrete extension of a $CCS$-normal space $CCS$-normal?

\item[2)] Is $A(X)$ $CCS$-normal, when $X$ is $CCS$-normal?
\end{itemize}

\end{document}